\documentclass[11pt]{amsart}
  \usepackage{amsfonts,amsthm,amssymb,eucal,color,enumerate,extarrows,graphicx}
  \usepackage[all]{xy}
  \usepackage[colorlinks,citecolor=magenta,linkcolor=black]{hyperref}
  \usepackage[a4paper,margin=1.2in]{geometry}  
 
  \numberwithin{equation}{section}

  \DeclareMathOperator{\Spec}{Spec}
  \DeclareMathOperator{\cok}{cok}
  \DeclareMathOperator{\im}{im}

  \newcommand{\cO}{\mathcal{O}}
  \newcommand{\Aff}{\mathbf{A}}
  \newcommand{\CC}{\mathbf{C}}
  \newcommand{\FF}{\mathbf{F}}
  \newcommand{\NN}{\mathbf{N}}
  \newcommand{\PP}{\mathbf{P}}
  \newcommand{\QQ}{\mathbf{Q}}
  \newcommand{\RR}{\mathbf{R}}
  \newcommand{\ZZ}{\mathbf{Z}}

  \newcommand{\isomto}{\xrightarrow{\,\smash{\raisebox{-0.65ex}{\ensuremath{\scriptstyle\sim}}}\,}}
  
  \newtheorem{thmi}{Theorem} 
  \newtheorem{thm}{Theorem}[section]
  \newtheorem{lemma}[thm]{Lemma}
  \newtheorem{prop}[thm]{Proposition}
  \newtheorem{cor}[thm]{Corollary}
  \newtheorem{conj}[thm]{Conjecture}

  \theoremstyle{definition} 
  \newtheorem{defin}[thm]{Definition}
  \newtheorem{rmk}[thm]{Remark}
  \newtheorem{example}[thm]{Example}

  \author[P.\ Achinger]{Piotr Achinger}
  \address{Institute of Mathematics of the Polish Academy of Sciences 
      \newline \indent ul.\ Śniadeckich 8, 00-656 Warsaw, Poland}
  \email{pachinger@impan.pl}

  \title[Hodge symmetry for rigid varieties via log hard Lefschetz]{Hodge symmetry for rigid varieties\\ via log hard Lefschetz}

\begin{document}

\begin{abstract}
  Motivated by a question of Hansen and Li, we show that a smooth and proper rigid analytic space $X$ with projective reduction satisfies Hodge symmetry in the following situations: (1) the base non-archimedean field $K$ is of residue characteristic zero, (2) $K$ is $p$-adic and $X$ has good ordinary reduction, (3) $K$ is $p$-adic and $X$ has ``combinatorial reduction.'' We also reprove a version of their result, Hodge symmetry for $H^1$, without the use of moduli spaces of semistable sheaves. All of this relies on cases of Kato's log hard Lefschetz conjecture, which we prove for $H^1$ and for log schemes of ``combinatorial type.''
\end{abstract}

\maketitle

\section{Introduction}

It is one of the basic consequences of Hodge theory that a compact K\"ahler manifold $X$ satisfies Hodge symmetry, i.e.
\[ 
  h^{i,j}(X) = h^{j,i}(X) 
  \quad \text{for all} \quad i, j \geq 0.
\]
Here $h^{i,j}$ denotes the dimension of the Dolbeault cohomology group $H^j(X, \Omega_X^i)$, where $\Omega^i_X$ is the $i$-th sheaf of holomorphic differential forms on $X$. As a consequence, the analogous statement holds for algebraic differential forms on smooth projective (or only proper) algebraic varieties over a field of characteristic zero. In fact, one can avoid Hodge theory and prove the latter fact using only algebraic methods, by means of Deligne's proof of the hard Lefschetz theorem and the Hodge--Tate decomposition in $p$-adic Hodge theory (see \cite[I.4.4]{FontaineMessing}).

This well-known link between hard Lefschetz and Hodge symmetry is very simple. For a~smooth projective variety $X$ over $\CC$, the isomorphisms induced by the Lefschetz operator
\[ 
  L^r \colon H^{n-r}_{\rm dR}(X)\isomto H^{n+r}_{\rm dR}(X)
\] 
are strictly compatible with the Hodge filtrations with appropriate shifts (i.e., inducing an isomorphism ${\rm Fil}^i\isomto {\rm Fil}^{i+r}$). This yields for $i+j=n-r$, $n=\dim X$:
\[ 
  \dim H^j(X, \Omega^i_X) = \dim {\rm gr}^j H^{n-r}_{\rm dR}(X) = \dim {\rm gr}^{j+r} H^{n+r}_{\rm dR}(X) =  \dim H^{r+j}(X, \Omega^{n-j}_X), 
\]
and by Serre duality, the right hand side equals $\dim H^i(X, \Omega^j_X)$. 

The recent years have seen a surge of interest in understanding to what extent the results in complex analytic geometry can be imported to rigid analytic geometry over some non-archimedean field $K$ \cite{ScholzeICM}. If we regard smooth proper rigid spaces as non-archimedean analogs of compact complex manifolds, then what should be the analog of the K\"ahler condition? The answer has been recently suggested by Li \cite{Li}, who studied the class of rigid spaces with \emph{projective reduction}, i.e.\ admitting a formal model whose special fiber is projective. The following result of Hansen and Li confirmed this expectation. 

\setcounter{thmi}{-1}
\begin{thmi}[{\cite[Theorem~1.2]{HansenLi}}] \label{thm:hansen-li}
  Let $X$ be a smooth proper rigid space over a~$p$-adic field $K$. Assume that $X$ has a formal model $\mathfrak{X}$ over $\cO_K$ whose special fiber is projective. Then
  \[
    h^{1,0}(X) = h^{0,1}(X).
  \]
\end{thmi}

Hansen and Li asked \cite[Question~1.3]{HansenLi} (also \cite[Conjecture~2.4]{ScholzeICM}) whether under the same assumptions one has $h^{i,j}(X) = h^{j,i}(X)$ for all $i,j\geq 0$. It turns out that their conjecture is not true as stated, but actually depends on some arithmetic properties of the reduction type. In this paper we prove that it holds in the following situation:

\begin{thmi}[Cases of Hodge symmetry] \label{thmi:HS}
  Let $K$ be a complete discretely valued field of characteristic zero with ring of integers $\cO_K$ and perfect residue field $k$. Let $X$ be a smooth proper rigid space over $K$ admitting a semistable model $\mathfrak{X}$ over $\cO_K$ whose special fiber $Y=\mathfrak{X}_0$ is projective. Then $h^{i,j}(X) = h^{j,i}(X)$ for $i,j\geq 0$ under one of the following conditions:
  \begin{enumerate}[\hphantom{xxx}(a)]
    \item \label{thmi:HS-char0} If $k$ has characteristic zero (Theorem~\ref{thm:HS-char-0}).
    \item \label{thmi:HS-good-ordinary} If $k$ has characteristic $p>0$ and $X$ has good ordinary reduction (Theorem~\ref{thm:HS-p-adic}(\ref{thm:HS-good-ordinary})).
    \item \label{thmi:HS-H1} If $i+j=1$, (almost) recovering Theorem~\ref{thm:hansen-li} (Theorem~\ref{thm:HS-p-adic}(\ref{thm:HS-H1})).
    \item \label{thmi:HS-comb-red} If arbitrary intersections of irreducible components of $Y$ have torsion-free crystalline cohomology and admit projective liftings to characteristic zero whose cohomology is generated by algebraic cycles and such that the relative Hodge cohomology is torsion-free (Theorem~\ref{thm:HS-p-adic}(\ref{thm:HS-comb-red})).
  \end{enumerate}
\end{thmi}

On the other hand, in \cite{Petrov}, Petrov constructs a stunning counterexample to the general question. It is obtained as the generic fiber of a quotient of a carefully chosen formal abelian variety over $\ZZ_p$ by a finite cyclic group action, and violates Hodge symmetry in degree $3$. 

Case (\ref{thmi:HS-comb-red}) of the above theorem encompasses varieties with ``totally degenerate'' or ``combinatorial'' reduction one often encounters in the context of mirror symmetry, e.g.\ ones where the special fiber $Y$ is a union of toric varieties glued along toric subvarieties (see Example~\ref{ex:cellular}). It could be worthwile to come up with examples of such $X$ which are not deformation equivalent to a projective variety.

The proofs of the above results rely on forms of the log hard Lefschetz theorem applied to the log special fiber of a semistable model of the rigid variety $X$. Let $k$~again be a~perfect field and let $Y$ be a strictly semistable log scheme over $k$ (Definition~\ref{def:strictly-semistable}) purely of dimension $n$. Let $H^*(Y)$ denote either log Betti cohomology $H^*(\tilde Y_{\rm log}, \QQ)$ (if $k=\CC$) , log $\ell$-adic cohomology $H^*_{\text{log-\'et}}(Y_{\bar s}, \QQ_\ell)$, or log crystalline cohomology $H^*_{\text{log-cris}}(Y/W(k))[1/p]$ (if $k$ has characteristic $p>0$), see \S\ref{ss:log-cohomology}. If $\mathcal{L}$ is an ample line bundle on $Y$, then cup product with powers of the logarithmic variant of its first Chern class $c_1(\mathcal{L})\in H^2(Y)$ yields maps 
\begin{equation} \label{eqn:intro-logHL} 
  L^r \colon H^{n-r}(Y) \to H^{n+r}(Y).
\end{equation}
Kato's log hard Lefschetz conjecture \cite[Conjecture~9.5]{Nakkajima} states that \eqref{eqn:intro-logHL} is an isomorphism for all $r\geq 0$. 

If $k=\CC$ and $H^*$ is log Betti cohomology, then the log hard Lefschetz conjecture is a theorem due to Nakkajima \cite[Theorem~9.14]{Nakkajima}, based on M.\ Saito's work on the monodromy filtration \cite[4.2.2]{SaitoMHP}. This in turn implies log Hodge symmetry for Y \cite[Corollary~9.15]{Nakkajima}. Using the base change results of \cite{IKN}, one can then deduce that if $Y$ is the special fiber of a semistable formal scheme $\mathfrak{X}$ over $\cO_K = \CC[\![t]\!]$, then the rigid generic fiber $X = \mathfrak{X}_K$ satisfies Hodge symmetry, proving~(\ref{thmi:HS-char0}). 

Suppose from now on that $K$ is $p$-adic and let $H^*(Y)$ denote log crystalline cohomology $H^*_{\text{log-cris}}(Y/W(k))[1/p]$. Assertions (\ref{thmi:HS-good-ordinary})--(\ref{thmi:HS-comb-red}) rely on the following observation.

\begin{thmi}[Log hard Lefschetz implies symmetry] 
  Let $X$ be a smooth proper rigid analytic space over $K$. Suppose that there exists a strictly semistable formal model $\mathfrak{X}$ of $X$ over $\cO_K$ whose log special fiber $Y = \mathfrak{X}_k$ admits an ample line bundle for which the log hard Lefschetz theorem holds in log crystalline cohomology in degree $q=n-r$ where $n=\dim X$. Then:
  \begin{enumerate}[\hphantom{xxx}(a)]
    \item (Proposition~\ref{prop:log-HL-implies-SS}) The Frobenius slopes $\alpha_0 \leq \ldots \leq \alpha_{s}$ ($s=\dim H^q(Y)-1$) on $H^q(Y)$ satisfy the symmetry
    \[ 
      \alpha_{s-k} = q - \alpha_k. 
    \]  
    \item (Proposition~\ref{prop:linear-rel-hij}) The Hodge numbers $h^{i,j}(X)$ with $i+j = q$ satisfy the relation
    \[ 
      \sum_{i+j = q} (i-j) \cdot h^{i,j}(X) = 0.
    \]
  \end{enumerate} 
  In particular, if $q\leq 2$ or if $Y$ is ordinary, then $h^{i,j}(X) = h^{j,i}(X)$ for $i+j=q$ (see Corollary~\ref{cor:ordinary-HS}).
\end{thmi}

The proof is relatively standard: log hard Lefschetz combined with Poincar\'e duality imply that the $F$-isocrystal $H^q(Y)$ admits an isomorphism $H^q(Y) \isomto H^q(Y)^\vee(-q)$, proving (a). Assertion (b) then follows from the Hyodo--Kato isomorphism 
\[ 
  H^q(Y)\otimes K\isomto H^q_{\rm dR}(X/K)
\] 
and the fact that $H^q(Y)$ endowed with the Hodge filtration on $H^q(Y)\otimes K \simeq H^q_{\rm dR}(X/K)$ is a weakly admissible filtered $(\varphi, N)$-module \cite{ColmezNiziol}, and in particular the endpoints of the Hodge and Newton polygons coincide. Then the $q\leq 2$ case follows directly from (b), and for the last assertion, ordinarity of $Y$ implies that the Newton polygon of $H^q(Y)$ equals the Hodge polygon of $H^q_{\rm dR}(X/K)$, so the slope symmetry (a) gives Hodge symmetry. (A~similar observation has been used in the context of Hodge symmetry for Hodge--Witt varieties in characteristic $p$ by Joshi \cite{Joshi} based on \cite{Ekedahl}, see Remark~\ref{rmk:joshi}.) 

Since hard Lefschetz (no log) holds for crystalline cohomology, we have thus obtained Theorem~\ref{thmi:HS}(\ref{thmi:HS-good-ordinary}). The remaining (\ref{thmi:HS-H1}) and (\ref{thmi:HS-comb-red}) rely on the following cases of log hard Lefschetz. 

\begin{thmi}[Cases of log hard Lefschetz] \label{thmi:log-HL}
  Let $Y$ be a strictly semistable log scheme over $k$, purely of dimension $n$. The log hard Lefschetz conjecture holds in the following situations.
  \begin{enumerate}[\hphantom{xxx}(a)]
    \item \label{thmi:log-HL-H1} For $q=n-r=1$ (Theorem~\ref{thm:log-HL-H1}).
    \item \label{thmi:log-HL-comb} If the cohomology of every intersection of irreducible components of $Y$ is generated by algebraic cycles and if the Hodge standard conjecture is satisfied for these intersections. (Corollary~\ref{cor:log-HL-comb}).
  \end{enumerate}
\end{thmi}

Result (\ref{thmi:log-HL-H1}) has been obtained previously by T.\ Kajiwara \cite[p.\ 167]{Nakkajima}, though the proof has never appeared. The proofs rely on the arguments of M.\ Saito \cite{SaitoMHP,SaitoPositivity}, Rapoport--Zink \cite{RapoportZink}, and T.\ Ito \cite{Ito} in the context of the weight-monodromy conjecture/theorem. 

The proof of Theorem~\ref{thm:hansen-li} relied on the earlier work of Li \cite{Li}, ingeniously bringing the moduli spaces of semistable sheaves into the argument. Our proof of Hodge symmetry for $H^1$, though slightly weaker (we need a semistable model, and it is not clear if one can use alterations to reduce to this case), completely circumvents this aspect of their proof, and is more ``motivic'' in a sense.

The above results motivate the following question, a non-archimedean analog of the search for restrictions on the homotopy type of a K\"ahler manifold. In \cite{AchingerTalpo}, the author together with M.\ Talpo have constructed a homotopy type $\Psi_{\rm rig}(X)$ associated to a smooth rigid analytic space $X$ over $\CC(\!(t)\!)$. What are the restrictions on $\Psi_{\rm rig}(X)$ for $X$ with projective reduction? (E.g.\ Theorem~\ref{thmi:HS}(\ref{thmi:HS-char0}) shows that the odd degree cohomology has even dimension.)

The paper is organized as follows. Sections \ref{s:logHL}--\ref{s:logHL-comb} are devoted to log hard Lefschetz: in \S\ref{s:logHL}, we formulate the conjecture and review the closely related weight spectral sequence, in \S\ref{s:logHL-H1}, we prove it for $H^1$, and in \S\ref{s:logHL-comb} for log varieties of combinatorial type. Then, in \S\ref{s:HS-char-0}, we prove Hodge symmetry for rigid spaces with projective reduction in equal characteristic zero. In the final \S\ref{s:p-adic}, we deal with the $p$-adic case.

\subsection*{Acknowledgements} 

The author is grateful to Shizhang Li, Sasha Petrov for many useful discussions and to Alex Youcis and Maciek Zdanowicz for comments on the manuscript. The overall strategy in this paper and some of its results were also independently observed by Petrov.

This work is a part of the project KAPIBARA supported by the funding from the European Research Council (ERC)
under the European Union’s Horizon 2020 research and innovation programme (grant agreement No 802787).


\section{Log hard Lefschetz and the weight spectral sequence}
\label{s:logHL}

Until the end of this paper, we fix a perfect field $k$. 

\subsection{Strictly semistable log schemes}
\label{ss:semistable-log}

We denote by $s = \Spec (k\to \NN)$ the standard log point over $k$. We use the notation $\underline Y$ for the underlying scheme of a log scheme $Y$. We refer to Ogus' book \cite{Ogus} for notation and terminology regarding log schemes.

\begin{defin} \label{def:strictly-semistable}
  A \emph{strictly semistable log scheme} over $k$ is a log scheme $Y$ over $s$ locally admitting a chart of the form $1\mapsto (1, \ldots, 1)\colon \NN\to \NN^r$ and such that the irreducible components of $\underline{Y}$ are smooth over $k$.
\end{defin}

In other words, the underlying scheme $\underline{Y}$ Zariski locally admits an \'etale morphism $\underline{Y}\to \Spec k[x_1, \ldots, x_r]/(x_1 \cdot\ldots \cdot x_r)$ and the log structure of $Y$ is obtained by pull-back from the log structure on the standard semistable scheme $\Spec k[x_1, \ldots, x_r, t]/(x_1 \cdot\ldots\cdot x_r -t)$ induced by the open subset $\{t\neq 0\}$. Thus, if $\mathfrak{X}$ is a strictly semistable (formal) scheme over a complete discrete valuation ring $\cO$ with residue field $k$, then the special fiber $Y = \mathfrak{X}_k$ endowed with the natural log structure is a strictly semistable log scheme.   Such log schemes are also called SNCL varieties in \cite{Nakkajima}. See also \cite[III 1.8]{Ogus}.

In the following, we will denote the irreducible components of $\underline{Y}$ by $Y_1, \ldots, Y_s$ and for $I\subseteq \{ 1, \ldots, s\}$, we write $Y_I = \bigcap_{i\in I} Y_i$ ($Y_\emptyset = \underline{Y}$ by convention). Note that if $\underline{Y}$ is everywhere of dimension $n$, then each $Y_I$ with $I\neq \emptyset$ is either empty or everywhere of dimension $n + 1 - |I|$. Finally, for $k\geq 1$ we denote by $Y^{(k)}$ the disjoint union of all $Y_I$ with $|I|=k$. The natural map $a_k\colon Y^{(k)}\to \underline{Y}$ identifies $Y^{(k)}$ with the normalization of $\bigcup_{|I|=k} Y_I$. The schemes $Y^{(\bullet+1)}$ naturally form a semi-simplicial scheme, and the maps $a_{\bullet + 1}$ provide an augmentation to $\underline Y$.

\subsection{Logarithmic cohomology theories}
\label{ss:log-cohomology}

We now briefly introduce the three cohomology theories for proper and strictly semistable log schemes over $k$: log Betti, log crystalline, and log $\ell$-adic cohomology. We refer e.g.\ to \cite[I \S 2]{Nakkajima} for details. 

Over $k=\CC$, one has the Kato--Nakayama space $Y_{\rm log}$ \cite[V]{Ogus}, functorially attached to any fs log scheme over $k$. Moreover, we have $s_{\rm log} \simeq S^1$. For $Y$ strictly semistable over $k$, we define $\tilde Y_{\rm log}$ to be the pull-back of $Y_{\rm log}\to s_{\rm log}$ along the universal covering $\exp(2\pi i z)\colon \RR\to S^1$. The log Betti cohomology groups $H^*(\tilde Y_{\rm log}, \QQ)$ are endowed with a~unipotent action of $\pi_1(S^1)\simeq \ZZ$ and enjoy the familiar properties of Betti cohomology, for example admit a natural isomorphism $H^*(\tilde Y_{\rm log}, \QQ) \otimes \CC\simeq H^*(Y, \Omega^\bullet_{Y/s})$. 

Suppose that $k$ has characteristic $p>0$. Hyodo and Kato \cite{HyodoKato} associate to $Y$ the log crystalline cohomology groups $H^i_{\text{log-cris}}(Y/W(k))$, endowed with the crystalline Frobenius $\varphi$ and the monodromy operator $N$ satisfying $N\varphi = p\varphi N$. We will mostly work with the associated $(\varphi, N)$-module $H^i_{\text{log-cris}}(Y/W(k))[1/p]$ over $K_0 = W(k)[1/p]$. 

Finally, in any characteristic one has the $\ell$-adic cohomology groups $H^i_{\text{log-\'et}}(Y_{\bar s}, \QQ_\ell)$ (with $\ell$ an auxiliary prime invertible in $k$). We will not use them in what follows.

If $\underline{Y}$ is smooth, each of these cohomology theories agrees with the corresponding classical cohomology (Betti, crystalline, $\ell$-adic) of $\underline{Y}$. If $F=H^0(Y)$ is the coefficient field, then $F(1):=H^2(\PP^1)^\vee$ is one-dimensional over $F$, fixing a basis of this vector space once and for all will allow us to ignore Tate twists when we do not need them. If $\mathcal{L}$ is a line bundle on $\underline{Y}$, then in each of the above cohomology theories there is a naturally defined first Chern class $c_1(\mathcal{L})\in H^2(Y)(1)$. 

\subsection{Line bundles and log hard Lefschetz}
\label{ss:log-HL-statement}

Pick as $H^*$ one of the cohomology theories in \S\ref{ss:log-cohomology}. Let $Y$ be a proper strictly semistable log scheme over $k$ and let $\mathcal{L}$ be a line bundle on $\underline{Y}$, and let $c_1(\mathcal{L}) \in H^2(Y)(1)$ be its first Chern class. Suppose that $\underline{Y}$ is proper and geometrically connected of dimension $n$. For $0\leq r\leq n$, cup product with $c_1(\mathcal{L})^r$ defines a map
\begin{equation} \label{eqn:loghl} 
  L^r \colon H^{n-r}(Y) \to H^{n+r}(Y)(r). 
\end{equation}

\begin{conj}[{Log hard Lefschetz conjecture}] \label{conj:log-HL}
  If $\mathcal{L}$ is ample, then \eqref{eqn:loghl} is an isomorphism for all $0\leq r\leq n$.
\end{conj}

We say that the log hard Lefschetz holds for $\mathcal{L}$ \emph{in degree $q$} if \eqref{eqn:loghl} is an isomorphism with $q=n-r$. 

\subsection{The weight spectral sequence}
\label{ss:weight-ss}

Let $k$ be an algebraically closed field and let $H^*$ be one of the cohomology theories listed in \S\ref{ss:log-cohomology}. The restriction of $H^*$ to smooth proper schemes is then a Weil cohomology theory satisfying the hard Lefschetz theorem. 

The weight spectral sequence expresses the logarithmic cohomology $H^*(Y)$ of a proper strictly semistable log scheme $Y$ over $k$ in terms of the ``classical'' cohomology $H^*(Y^{(k)})$. It has been constructed for Betti cohomology (over $k=\CC$) by Steenbrink \cite{Steenbrink,SteenbrinkLog}, for $\ell$-adic \'etale cohomology by Rapoport and Zink \cite{RapoportZink}, and for crystalline cohomology by Mokrane \cite{Mokrane} (see also \cite{Nakkajima}). It has the form 
\[ 
  E^{a,b}_1 = \bigoplus_{k\geq \max(a, 0)} H^{2(a-k)+b}(Y^{(2k-a+1)})
  \quad \Rightarrow \quad
  H^{a+b}(Y), 
\]
where the differentials $d_1^{a,b} \colon E^{a,b}_1 \to E^{a+1, b}_1$ are as in the paragraph below. The spectral sequence degenerates at the $E_2$-page. 

To describe the differentials $d_1^{a,b}$, we shall consider the maps
\[ 
  \rho \colon H^i(Y^{(k)}) \to H^i(Y^{(k+1)})
\] 
induced by the alternating sum of the inclusion maps and
\[ 
  \tau \colon H^i(Y^{(k)}) \to H^{i+2}(Y^{(k-1)})
\]
induced by the Gysin homomorphisms, cf.\ \cite[\S 5]{Ito}. These induce maps 
\[
  \rho \colon E^{a,b}_1\to E^{a+1, b}_1 \quad \text{(increasing the index $k$ by $1$)}
\]
\[
  \tau\colon E^{a,b}_1 \to E^{a+1, b}_1 \quad  \text{(preserving the degree $k$).}
\]
Then $d^{a,b}_1 = \rho + \tau$. We shall need the following relations: $\rho^2 = 0$, $\tau^2 = 0$, $\rho \tau + \tau\rho = 0$.

We denote by $N$ the natural maps $N\colon E_1^{a,b} \to E_1^{a+2, b-2}$ increasing the index $k$ by $2$, which are either zero or identity on the summands. They commute with the operators $\rho$, $\tau$, and $d_1$, and hence induce maps $N\colon E_2^{a,b}\to E_2^{a+2, b-2}$. The \emph{weight-monodromy conjecture} for $Y$ is the statement that the maps
\[ 
  N^r \colon E_2^{-r, w+r} \to E_2^{r,w-r}
\] 
are isomorphisms. It is false for general $Y$ \cite[\S 6]{Nakkajima}, but it is expected to hold if $\underline{Y}$ is projective or if $Y$ is the special fiber of a proper semistable scheme over a complete discrete valuation ring. We shall prove, in Proposition~\ref{prop:WM-H1} below, that if $\underline{Y}$ is projective, then the above statement holds with $r=w=1$. 

We fix an ample line bundle $\mathcal{L}$ on $\underline{Y}$ and denote by $L\colon H^*(Y^{(k)}) \to H^{*+2}(Y^{(k)})$ the Lefschetz operators, i.e.\ the cup product with $c_1(\mathcal{L})$. They commute with the maps $\rho$, $\tau$, and $N$ and hence induce a degree $(0,2)$ map $E_*^{a,b}\to E_*^{a, b+2}$ which commutes with $N$. Thanks to the $E_2$-degeneration, the log hard Lefschetz conjecture (Conjecture~\ref{conj:log-HL}) holds for $(Y, \mathcal{L})$ if and only if the maps
\begin{equation} \label{eqn:logHL-Weil} 
  L^r \colon E_2^{n-b-r,b} \to E_2^{n-b+r,b}
\end{equation}
are isomorphisms. 

The hard Lefschetz theorem for $H^*$ restricted to smooth projective schemes implies that the pairing
\[ 
  \langle, \rangle \colon H^q(X^{(k)}) \times H^q(X^{(k)}) \to F,
  \quad
  \langle x,  y\rangle = L^{n+1-k-q} x\cdot y
\]
is non-degenerate. Further, the maps $\rho$ and $\tau$ are adjoint in the following sense. For $x\in H^q(X^{(k)})$, $y\in H^{q'}(X^{(k-1)})$, $q+q'=2s$:
\[ 
  L^{n-k+1-s} x \cdot \rho y = \pm L^{n-k+1-s} \tau x \cdot y,
\]
and similarly for $x\in H^q(X^{(k)})$, $y\in H^{q'}(X^{(k+1)})$, $q+q'=2s$:
\[ 
  L^{n-k-s} x \cdot \tau y = \pm L^{n-k-s} \rho x \cdot y,
\]
see \cite[(5.1)]{Ito}. We will frequently make use of these relations.

Using the Hodge-theoretic argument \cite[4.2.2]{SaitoPositivity}, Nakkajima has obtained:

\begin{thm}[{\cite[Theorem~9.14]{Nakkajima}}] \label{thm:log-HL-char0}
  The log hard Lefschetz conjecture holds if $k=\CC$ for each of the cohomology theories listed in \S\ref{ss:log-cohomology}.
\end{thm}


\section{\texorpdfstring{Log hard Lefschetz and weight-monodromy for $H^1$}{Log hard Lefschetz and weight-monodromy for H1}}
\label{s:logHL-H1}

In this section, we shall prove the log hard Lefschetz theorem for $H^1$. For this, we will need to show that the weight--monodromy conjecture holds for $H^1$. Both of these results have been obtained earlier by T.\ Kajiwara (unpublished, cf.\ \cite[p.\ 167]{Nakkajima}). We let $H^*$ be one of the cohomology theories listed in \S\ref{ss:log-cohomology}.

The proof of the case of weight--monodromy is a rather straightforward adaptation of the arguments of Ito \cite[\S 5]{Ito}, which are themselves based on \cite[\S 4]{SaitoMHP}; we can do this without the strong assumptions in \emph{op.\ cit.} because as a consequence of the Hodge index theorem, the Hodge standard conjecture holds for divisors. 

For the log hard Lefschetz, we will also need the fact that the Lefschetz operator $L^{n-1}\colon H^1\to H^{2n-1}$ on the cohomology of a smooth projective variety of dimension $n$ is ``motivic,'' i.e.\ induced by a polarization of the Albanese variety. This follows from the argument given by Kleiman \cite[2A9.5]{Kleiman}.

In the following, we fix a strictly semistable log scheme $Y$ over $k$, which is purely of dimension $n$, and an ample line bundle $\mathcal{L}$ on $\underline{Y}$. We continue using the notation of \S\ref{ss:weight-ss}, in particular we use heavily the pairing $\langle , \rangle$ on the cohomology groups $H^*(Y^{(k)})$ induced by $\mathcal{L}$ and the maps $\rho$ and $\tau$. We denote by $P^q(Y^{(k)})$ the \emph{primitive cohomology} of $Y^{(k)}$, i.e.\ the kernel of $L^{n-k+1-q+1}\colon H^q(Y^{(k)})\to H^{2(n-k+1)-q+2}(Y^{(k)})$. 

\begin{lemma} \label{lemma:ito-l5.5-H0}
The restriction of $\langle , \rangle$ on $H^0 (Y^{(k)})$ to $\im(\rho\colon H^0(Y^{(k-1)})\to H^0(Y^{(k)}))$ and to $\ker(\rho\colon H^0(Y^{(k)})\to H^0(Y^{(k+1)})$ are non-degenerate for all $k$.
\end{lemma}

\begin{proof}
The commutative diagram
\[ 
  \xymatrix{
    H^0(Y^{(k-1)}, \QQ) \ar[r]^-{\rho_\QQ} \ar[d] & H^0(Y^{(k)}, \QQ) \ar[r]^-{\rho_{\QQ}} \ar[d] & H^0(Y^{(k+1)}, \QQ) \ar[d] \\
    H^0(Y^{(k-1)}) \ar[r]^-\rho & H^0(Y^{(k)}) \ar[r]^-\rho & H^0(Y^{(k+1)}) 
       }
\]
provides natural $\QQ$-structures for the $F$-vector spaces in the bottom row. It therefore suffices to show that the induced pairing on $H^0(Y^{(k)}, \QQ)$ is non-degenerate when restricted to $\ker(\rho_\QQ)$ and $\im(\rho_{\QQ})$. But, if $Z$ is a connected component of $Y^{(k)}$, and $\chi_Z\in H^0(Y^{(k)}, \QQ)$ is its characteristic function, then $\langle \chi_Z, \chi_Z \rangle$ equals the top intersection number $c_1(\mathcal{L}|_Z)^{n+1-k}$, which is positive since $\mathcal{L}$ is ample. Since such elements $\chi_Z$ provide a basis of $H^0(Y^{(k)}, \QQ)$, the pairing $\langle, \rangle$ on this space is thus positive definite. Therefore its restriction to any subspace remains non-degenerate.
\end{proof}

\begin{lemma}[{cf.\ \cite[Lemma~5.5]{Ito}}] \label{lemma:ito-l5.5}
The restriction of $\langle, \rangle$ on $H^2(Y^{(1)})$ to 
\[
  \im\left(\tau\colon H^0(Y^{(2)})\to H^2(Y^{(1)})\right)\cap P^2(Y^{(1)})
\]
is non-degenerate.
\end{lemma}

\begin{proof}
We have the following commutative square
\[ 
  \xymatrix{
    H^0(Y^{(2)}, \QQ) \ar[r]^{\tau_\QQ} \ar[d] & {\rm NS}(Y^{(1)}) \otimes \QQ \ar[d] \\
    H^0(Y^{(2)}) \ar[r]_\tau & H^2(Y^{(1)}), 
  }
\]
where $H^0(Y^{(2)}, \QQ) = \QQ^{\pi_0(Y^{(2)})}$. Here, the top map $\tau_{\QQ}$ is defined as the Gysin map for cycles. The vertical maps are compatible with the natural pairings $\langle, \rangle$ induced by the fixed line bundle $\mathcal{L}$. Since $\langle, \rangle$ is non-degenerate on ${\rm NS}(Y^{(1)})\otimes \QQ$ by the Hodge index theorem, the right vertical map induces an injection
\[ 
   {\rm NS}(Y^{(1)}) \otimes F \hookrightarrow H^2(Y^{(1)}).  
\]
It follows that $\ker(\tau_\QQ)\otimes F \to \ker(\tau)$ is an isomorphism. In other words, we have shown that the natural $\QQ$-structure on $H^0(Y^{(2)})$ induces a $\QQ$-structure on $\im(\tau)$.

 It now suffices to show that the pairing in the top right corner remains non-degenerate when restricted to $\im(\tau_{\QQ})\cap (P^2_\QQ)$ where $P^2_\QQ$ is the preimage in ${\rm NS}(Y^{(1)})\otimes \QQ$ of $P^2 (Y^{(1)})$. Note again that $\im(\tau_{\QQ})\cap P^2_\QQ$ is a $\QQ$-structure for $\im(\tau)\cap P^2 (Y^{(1)})$. 

But, again by the Hodge index theorem, $\langle, \rangle$ is definite on $P^2_\QQ$, and hence so is its restriction to $\im(\tau_{\QQ})$.
\end{proof}

\begin{lemma}[{cf.\ \cite[Lemma~5.6]{Ito}}] \label{lemma:ito-l5.6}
  The image of the composition
  \[
    \tau \rho\colon H^0(Y^{(1)}) \to H^0(Y^{(2)})\to H^2(Y^{(1)})
  \]
  is the orthogonal complement to $\im(\tau)\cap P^2 (Y^{(1)})$ inside $\im(\tau)$.  
\end{lemma}

\begin{proof}
Given Lemmas~\ref{lemma:ito-l5.5-H0} and \ref{lemma:ito-l5.5}, it suffices to quote the proof of the first assertion of \cite[Lemma~5.6]{Ito} and \cite[Lemma~5.7]{Ito} with $i=0$ and $k=1$. Indeed, using Ito's notation, we have 
\[ 
  \im(\rho\colon H^0(Y^{(1)}) \to H^0(Y^{(2)})) = {\rm Im}^0 \rho_0^{(1)},
  \quad
  {\rm Im}^0 \tau_0^{(2)} = \im(\tau)\cap P^2(Y^{(1)}),
\]
and ${\rm Im}^1 \tau_0^{(2)} = \im(\tau)/(\im(\tau)\cap P^2(Y^{(1)}))$. The argument in the proof of \cite[Lemma~5.6]{Ito} shows that $\tau\colon \im(\rho)\to \im(\tau)/(\im(\tau)\cap P^2(Y^{(1)})$ is injective. It is then also bijective by \cite[Lemma~5.4]{Ito}. Finally, we check that $\im(\tau\rho)$ is orthogonal to the entire $\im(\tau)$. If $x = \tau\rho z$ and $y= \tau w$, then
\[ 
  \langle y, x\rangle = L^{n-2} y \cdot x = L^{n-2} \tau w \cdot \tau\rho z = \pm L^{n-2} w\cdot \rho \tau \rho z = 0.
\]
We conclude by Lemma~\ref{lemma:ito-l5.5}.
\end{proof}

\begin{lemma}[{cf.\ \cite[Lemma~5.1]{Ito}}] \label{lemma:ito-l5.1}
  In the following sequence of maps 
    \[ 
      H^0 (Y^{(1)}) \xlongrightarrow{\rho} H^0(Y^{(2)}) \xlongrightarrow{\tau} H^2(Y^{(1)}) \xlongrightarrow{\rho} H^2(Y^{(2)})
  \]
  one has $\ker(\tau)\cap \im(\rho) =0 \subseteq H^0(Y^{(2)}) $ and $\ker(\rho)\cap \im(\tau) = \im (\tau\rho) \subseteq H^2(Y^{(1)})$.
\end{lemma}

\begin{proof}
Let $0\neq x \in \ker(\tau)\cap \im(\rho) \subseteq H^0(Y^{(2)})$. Since $\langle, \rangle$ is non-degenerate on $\im(\rho) \subseteq H^0(Y^{(2)})$ by Lemma~\ref{lemma:ito-l5.5-H0}, there exists a $z\in H^0(Y^{(1)})$ such that $\langle x, \rho y \rangle \neq 0$, but
\[ 
  \langle x, \rho y \rangle  = L^{n-1} x \cdot \rho z = \pm L^{n-1} \tau x \cdot z = 0, 
\]
contradiction. Thus $\ker(\tau)\cap \im(\rho) = 0$.

Let now $0\neq x \in \ker(\rho)\cap \im(\tau) \subseteq H^2(Y^{(1)})$. By Lemma~\ref{lemma:ito-l5.6}(b), it suffices to show that $x$ is orthogonal to every element in $\im(\tau)\cap P^2(Y^{(1)})$. In fact, we have $\langle x, y\rangle = 0$ for every $y\in \im(\tau)$, for if $y=\tau z$ then
\[ 
  \langle x, y\rangle = \langle x, \tau z\rangle = L^{n-2} x \cdot \tau z = \pm L^{n-2} \rho x \cdot z = 0. \qedhere
\] 
\end{proof}

\begin{prop}[{Weight--monodromy conjecture for $H^1(Y)$}] \label{prop:WM-H1}
  The identity on $H^0(Y^{(2)})$ induces an isomorphism
  \[ 
    N\colon \ker\left({\textstyle (\tau+\rho)\colon H^0(Y^{(2)}) \to H^2(Y^{(1)})\oplus H^0(Y^{(3)})}\right)
    \isomto 
    \frac{ \ker(\rho\colon H^0(Y^{(2)})\to H^0(Y^{(3)})) }{\im(\rho\colon H^0(Y^{(1)})\to H^0(Y^{(2)})) }.
  \]
\end{prop}

\begin{proof}
We first show injectivity. Let $x\in H^0 (Y^{(2)})$ with $\tau x = 0$ and $\rho x = 0$, and suppose that $x= \rho y$. By the first assertion of Lemma~\ref{lemma:ito-l5.1} applied to $x\in \im \rho\cap \ker \tau$, we have $x = 0$. 

For the surjectivity, suppose that $x \in H^0(Y^{(2)})$ with $\rho x = 0$. Apply the second assertion of Lemma~\ref{lemma:ito-l5.1} to $\tau x \in H^2 (Y^{(1)})$, obtaining a $z\in H^0(Y^{(1)})$ such that $\tau x = \tau \rho z$. Let $y = x - \rho z$, so $\tau y = 0$ and $x = y + \rho z$ shows that $x$ is congruent modulo $\im \rho$ to an element of $\ker \tau$.
\end{proof}
  
\begin{thm} \label{thm:log-HL-H1}
  Let $Y$ be a proper strictly semistable log scheme over $k$. Then for every ample line bundle $\mathcal{L}$ on $Y$, the log hard Lefschetz conjecture holds for $H^1(Y)$. 
\end{thm}

\begin{proof}
Let $n=\dim Y$. The maps $\ell = L^{n-1}$ induce isomorphisms on the relevant entries of the $E_1$ page of the weight spectral sequence for $Y$ as follows:
\[ 
  \def\arhl{\ar@/^1em/@{.>}[uuuu]_{\rotatebox{90}{$\sim$}}^-{\ell} }
  \xymatrix@R=.7em@C=1.2em{ 
    H^{2n-r}(Y^{(3)}) \ar[r] & H^{2n-2}(Y^{(2)}) \ar[r] & H^{2n}(Y^{(1)}) & 0 \\
   & H^{2n-3}(Y^{(2)}) \ar[r] & H^{2n - 1}(Y^{(1)}) & 0 \\
   & & H^{2n-2}(Y^{(1)})\oplus H^{2n-4}(Y^{(3)}) \ar[r] & H^{2n-2}(Y^{(2)}) & 0\\
   & \cdots & & \cdots \\ 
   0 & H^0(Y^{(2)}) \arhl \ar[r] & H^2(Y^{(1)})\oplus H^0(Y^{(3)}) & \\
   & 0 & H^1(Y^{(1)}) \ar[r] \arhl & H^1(Y^{(2)}) \\
   & 0 & H^0(Y^{(1)}) \ar[r] & H^0(Y^{(2)}) \arhl \ar[r] & H^0(Y^{(3)})
  }
\]
The spectral sequence degenerates on the $E_2$ page. We need to show that the three maps $\ell$ induce isomorphisms on the $E_2$ page:
\[ 
  \textstyle \ell_0 \colon \frac{\ker(H^0(Y^{(2)}) \to H^0(Y^{(3)}))}{\im(H^0(Y^{(1)}) \to H^0(Y^{(2)}))}
  \longrightarrow
  \cok\left(H^{2n-2}(Y^{(1)})\oplus H^{2n-4}(Y^{(3)}) \to H^{2n-2}(Y^{(2)})\right)
\]
\[ 
  \textstyle \ell_1\colon \ker\left(H^1(Y^{(1)}) \to H^1(Y^{(2)})\right) 
  \longrightarrow
  \cok\left(H^{2n-3}(Y^{(2)}) \to H^{2n - 1}(Y^{(1)})\right), 
\]
\[ 
  \textstyle \ell_2 \colon \ker\left(H^0(Y^{(2)}) \to H^2(Y^{(1)})\oplus H^0(Y^{(3)})\right)
  \longrightarrow
  \frac{\ker(H^{2n-2}(Y^{(2)}) \to H^{2n}(Y^{(1)}) )}{\im(H^{2n-2}(Y^{(3)}) \to H^{2n-2}(Y^{(2)}))}
\]

We shall first prove that the map $\ell_2$ is injective. Say $y\in H^0(Y^{(2)})$ satisfies $\rho y =0$ and $\tau y =0$, and suppose that $L^{n-1} y = \tau z$ for some $z \in H^{2n-4}(Y^{(3)})$. If $y\neq 0$, then by Lemma~\ref{lemma:ito-l5.5-H0} there exists a $t\in H^0(Y^{(2)})$ such that $\rho t = 0$ and $\langle t, y \rangle \neq 0$. But
\[ 
  \langle t, y \rangle = t\cdot L^{n-1} y = t\cdot \tau z  = \pm \rho t\cdot z = 0,
\]
a contradiction.

Proposition~\ref{prop:WM-H1} implies that in the commutative square 
\[ 
  \xymatrix{
      \frac{\ker(H^{2n-2}(Y^{(2)}) \to H^{2n}(Y^{(1)}) )}{\im(H^{2n-4}(Y^{(3)}) \to H^{2n-2}(Y^{(2)}))} \ar[r]^-\sim_-N & \scriptstyle \cok\left(H^{2n-2}(Y^{(1)})\oplus H^{2n-4}(Y^{(3)}) \to H^{2n-2}(Y^{(2)})\right) \\ 
      \scriptstyle \ker\left(H^0(Y^{(2)}) \to H^2(Y^{(1)})\oplus H^0(Y^{(3)})\right) \ar[r]^-\sim_-N \ar[u]^{\ell_2} & \frac{\ker(H^0(Y^{(2)}) \to H^0(Y^{(3)}))}{\im(H^0(Y^{(1)}) \to H^0(Y^{(2)}))} \ar[u]_{\ell_0} 
  }
\]
the horizontal maps are isomorphisms. Moreover, by Poincar\'e duality, the opposite corners are dual to each other, and therefore all four vector spaces have the same dimension. Since $\ell_0$ is injective, it is thus an isomorphism, which implies that $\ell_2$ is an isomorphism as well.

Finally, we deal with $\ell_1$. The proof is directly inspired by \cite[Satz~2.13]{RapoportZink}. Let $A = {\rm Alb}(Y^{(1)})$ and $B = {\rm Alb}(Y^{(2)})$. The simplicial maps $d^0, d^1 \colon Y^{(2)} \to Y^{(1)}$ induce a map $\rho = d^0 - d^1 \colon B\to A$. Let $C$ be its cokernel. Then the source $\ker\left(H^1(Y^{(1)}) \to H^1(Y^{(2)})\right)$ of $\ell_1$ is identified with 
\[ 
  \ker\left(\rho^* \colon H^1(A) \to H^1(B)\right) = H^1(C),
\]
and similarly the target of $\ell_1$ is
\[ 
  \cok((\rho^\vee)^*\colon H^1(B^\vee) \to H^1(A^\vee)) = H^1(C^\vee).
\]
By Lemma~\ref{lemma:polarization} below, the map $\ell \colon H^1(A)\to H^1(A^\vee)$ is, up to replacing $\mathcal{L}$ with a multiple, induced by an isogeny $\Psi\colon A^\vee\to A$. Therefore $\ell_1$, interpreted as a map $H^1(C)\to H^1(C^\vee)$, is induced by the composition $C^\vee \to A^\vee \to A \to C$. This composition is an isogeny and hence $\ell_1$ is an isomorphism.
\end{proof}

\begin{lemma} \label{lemma:polarization}
  Let $X$ be smooth and projective over $k$, with an ample line bundle $\mathcal{L}$. Let $n=\dim X$. Let $A$ be the Albanese variety of $X$ and let $A^\vee$ be its dual, naturally identified with the Picard variety ${\rm Pic}^0(X)_{\rm red}$ of $X$. Then, after replacing $\mathcal{L}$ with $\mathcal{L}^m$ for some $m\geq 1$, there exists an isogeny $\Psi_{\mathcal{L}}\colon A^\vee \to A$ making the following square commute
  \[ 
    \xymatrix{ 
      H^1(X) \ar[rr]^{L^{n-1}}_\sim & & H^{2n-1}(X) \ar[d]^\simeq \\
      H^1(A) \ar[u]^\simeq  \ar[rr]_{\Psi_{\mathcal{L}}^*} & & H^1(A^\vee).
    }
  \]
\end{lemma}

\begin{proof}
We may assume that $\mathcal{L}$ is very ample. Let $i\colon Y\hookrightarrow X$ be a smooth one-dimensional linear section of $X$. Let $J = {\rm Alb}(Y) = {\rm Pic}^0(Y)$. By \cite[2A9.5]{Kleiman} we have a commutative diagram
\[ 
  \xymatrix{
    H^1(X) \ar[dr]_{i^*} \ar[rr]^{L^{n-1}}_\sim & & H^{2n-1}(X)\ar[d]^\simeq \\
    H^1(A)\ar[u]^\simeq \ar[dr]_{{\rm Alb}(i)^*} & H^1(Y) \ar@{=}[d] \ar[ur]_{i_*} & H^1(A^\vee) \\
    & H^1(J) \ar[ur]_{{\rm Pic}(i)^*}
  }
\]
We conclude that we can take for $\Psi$ the composition
\[ 
  A^\vee = {\rm Pic}^0(X)_{\rm red}
  \xrightarrow{{\rm Pic}^0_{\rm red}(i)} {\rm Pic}^0(Y) = J = {\rm Alb}(Y) 
  \xrightarrow{{\rm Alb}(i)}
  {\rm Alb}(X) = A. \qedhere 
\]
\end{proof}


\section{Log hard Lefschetz for varieties of combinatorial type}
\label{s:logHL-comb}

\subsection{Ito modules}
\label{ss:ito-modules}

Consider a bigraded vector space $V = \bigoplus_{i,j\in\ZZ} V^{i,j}$ over $\QQ$ endowed with commuting operators $N$ (of bidegree $(2,0)$), $L$ (of bidegree $(0,2)$), and $d$ (of bidegree $(1,1)$) satisfying $d^2 = 0$. We assume that
\[ 
  N^i \colon V^{-i,j} \to V^{i,j}
  \quad \text{and} \quad 
  L^j \colon V^{i,-j} \to V^{i,j}
\]
are isomorphisms for $i,j\geq 0$.

Suppose moreover that $V$ is endowed with a pairing $\langle, \rangle$ such that $V^{i,j}$ and $V^{i',j'}$ are orthogonal unless $i+i' = 0 = j+j'$ and inducing a perfect pairing betwen $V^{i,j}$ and $V^{-i,-j}$. We assume that $\langle, \rangle$ satisfies the relations
\[ 
  \langle x, y \rangle = \pm \langle y, x \rangle, 
  \quad
  \langle \square x, y \rangle = \pm \langle x, \square y \rangle
  \quad\text{for} \quad
  \square \in \{ N, L, d\}
\]
for homogenous elements $x,y\in V$, where the unspecified signs depend on their bigradings.

We also make the following positivity assumption: set 
\[
  {}_0 V^{-i,-j} = \ker N^{i+1}\cap \ker L^{j+1} \subseteq V^{-i,-j}
  \quad \text{for} \quad 
  i,j\geq 0.
\]
Then the pairing
\[ 
  \langle - , N^i L^j - \rangle 
  \quad \text{on} \quad
  {}_0 V^{-i,-j} \times {}_0 V^{-i,-j} 
\]
is either positive or negative definite, again depending on the bigrading.

If we fix an integer $n$ and suppose that $V^{i,j} = 0$ if $i+j+n$ is odd. If we regard each $V^{i,j}$ as a pure Hodge structure of weight $i+j+n$ purely of type $((i+j+n)/2, (i+j+n)/2)$, then $V\otimes \CC$ becomes a polarized differential bigraded module of type H--L (Hodge--Lefschetz) of weight $n$ as defined in \cite[4.2.1]{SaitoMHP}. To indicate the (implicit) use of such structures in \cite{Ito}, we will call such objects $V=(V, N, L, d, \langle,\rangle, n)$ \emph{Ito modules} (of weight $n$).

We can then apply \cite[4.2.2]{SaitoMHP} to obtain the following result.

\begin{thm} \label{thm:HV-is-Ito}
  Let $V$ be an Ito module of weight $n$, and set $H(V) = \ker d /\im d$. Then $N$, $L$ descend to $H(V)$, the pairing $\langle, \rangle$ induces a pairing on $H(V)$. Endowed with those and with zero differential $d$, $H(V)$ is again an Ito module of weight $n$.
\end{thm}

\subsection{Recall on algebraic cycles}
\label{ss:alg-cycles}

Let $Z$ be a smooth and proper variety over $k$ and let $H^*$ be a Weil cohomology theory. We denote by $C^i(Z)$ the $\QQ$-vector space of codimension $i$ cycles on $Z$, by 
\[ 
  {\rm cl}_Z \colon C^i(Z) \to H^{2i}(Z)
\]
the cycle class map, by $A^i(Z) \subseteq H^{2i}(Z)$ its image, and finally by $C^i_{\rm num}(Z)$ the quotient by $C^i(Z)$ by the group of cycles numerically equivalent to zero, i.e.\ pairing trivially with the entire $C^{n-i}(Z)$ where $n=\dim Z$. We say that \emph{$H^*(Z)$ is generated by algebraic cycles} if $H^i(Z)=0$ for $i$ odd and if $A^i(Z)$ generates $H^{2i}(Z)$ as a vector space over $F=H^0({\rm pt})$ for all $i$. If this is the case, then homological and numerical equivalence on $Z$ agree, i.e.\ $A^i(Z)\isomto C^i_{\rm num}(Z)$, and moreover $A^i(Z)\otimes F\isomto H^{2i}(Z)$ \cite[Proposition~3.6]{Kleiman}. The \emph{Hodge standard conjecture} is the statement that, for an ample line bundle $\mathcal{L}$ on $Z$, the restriction of the induced pairing $\langle , \rangle$ on $H^{2i}(Z)$ to the primitive algebraic classes $P^{2i}(Z)\cap A^i(Z)$ is positive definite for $i$ even and negative definite for $i$ odd \cite[\S 3]{Kleiman}.

\subsection{Varieties of combinatorial type}
\label{ss:comb-type}

Consider now a strictly semistable log scheme $Y$ over $k$, purely of dimension $n$, endowed with an ample line bundle $\mathcal{L}$. We fix a cohomology theory $H^*$ as in \S\ref{ss:weight-ss}. Suppose that each connected component $Z$ of every $Y^{(k)}$ satisfies the following two conditions:
\begin{enumerate}
  \item[(A)] $H^*(Z)$ is generated by algebraic cycles, 
  \item[(B)] $Z$ endowed with $\mathcal{L}|_Z$ satisfies the Hodge standard conjecture.  
\end{enumerate}
Consider the bigraded $\QQ$-vector space
\[ 
  V = \bigoplus_{i,j\in \ZZ} V^{i,j},
  \quad
  V^{i,j} =
  \begin{cases} 
    \bigoplus_{k\geq 0} A^{i-k + \frac{n-i+j}{2}}(Y^{(2k-i+1)}) & n-i+j \text{ even} \\
    0 & n-i+j \text{ odd}.
  \end{cases}
\]
By (A), this is a $\QQ$-structure for the $E_1$-page of the weight spectral sequence for $Y$:
\[ 
  V^{i,j} \otimes F \isomto E_1^{i,n-i+j}.
\]
The maps $N$, $L$, and $d$ on the $E_1$-page are rational with respect to this $\QQ$-structure and induce familiar operators on $V$, defined exactly as in the recipe in \eqref{ss:weight-ss}. Consequently, $\ker d / \im d$ on $V$ is similarly a $\QQ$-structure for the $E_2$-page. 

Moreover, Poincar\'e duality identifies $V^{-i,-j}$ with the dual of $V^{i,j}$, yielding a pairing $\langle, \rangle$ on $V$.

\begin{prop} \label{prop:V-is-ito}
  In the above situation, $V$ is an Ito module of weight $n$.
\end{prop}

\begin{proof}
Everything except the positivity assumption is clear. For the remaining statement, it suffices to note that for $i,j\geq 0$ with $i+j-n$ even, 
\[
  {}_0 V^{-i,-j} = \ker N^{i+1}\cap \ker L^{j+1} \subseteq V^{-i,-j}
\] 
is equal to $P^{(n-i-j)/2}(Y^{(i+1)})$, with the pairing $\langle -, N^i L^j -\rangle$ equal to the pairing $\langle -, L^j - \rangle$ up to sign (note $\dim Y^{i+1} = n-i$), as the latter pairing is positive definite by (B).
\end{proof}

Combining Proposition~\ref{prop:V-is-ito} with Theorem~\ref{thm:HV-is-Ito}, we obtain. 

\begin{cor} \label{cor:log-HL-comb}
  For $Y$ and $L$ satisfying (A) and (B), the log hard Lefschetz conjecture and the weight-monodromy conjecture hold for $Y$.
\end{cor}

\begin{rmk} \label{rmk:ito}
The second assertion has been obtained by Ito \cite[Proposition~5.1]{Ito}. He proved \cite[\S 4]{Ito} that if $Y$ is a certain special fiber of a model of a rigid variety $X$ uniformized by the Drinfeld upper half space, then the assumptions (A) and (B) are satisfied. This is not interesting from our point of view since such an $X$ is projective. 
\end{rmk}


\section{Hodge symmetry in equal characteristic zero}
\label{s:HS-char-0}

In this section, we let $k$ be a field of characteristic zero.

\begin{prop}[{\cite[Theorem~7.1 and Corollary~7.2]{IKN}}] \label{prop:ikn-relative-log-hodge}
  Let $f\colon X\to S$ be a proper, log smooth, and exact morphism of fs log schemes over $k$, and suppose that each stalk of $\mathcal{M}_S/\cO_S^\times$ is a free monoid. Then the relative log Hodge cohomology sheaves
  \[ 
    R^j f_* \Omega^i_{X/S}
  \]
  are locally free, with formation commuting with base change along every morphism $S'\to S$ of fs log schemes, for all $i,j\geq 0$.
\end{prop}

\begin{cor} \label{cor:hodge-free-formal}
  Let $\cO\simeq k[\![t]\!]$ be a complete dvr with residue field $k$, and let $\mathfrak{X}$ be a~semistable formal scheme over $\cO$, endowed with the standard log structure. Then the relative log Hodge cohomology groups $H^j(\mathfrak{X}, \Omega^i_{\mathfrak{X}/\cO})$ are free $\cO$-modules of finite rank, with formation commuting with the base change to $\Spec k$.
\end{cor}

\begin{proof}
Apply Proposition~\ref{prop:ikn-relative-log-hodge} to the base change $\mathfrak{X}_n$ of $\mathfrak{X}$ to $S_n = \Spec \cO/(t^{n+1})$ for all $n$ (note that the resulting map is automatically exact).
\end{proof}

\begin{thm}[{Log Hodge symmetry over $\CC$, \cite[Corollary~9.15]{Nakkajima}}] \label{thm:log-HS-char0}
  Let $Y$ be a strictly semistable log variety over $k$ such that $\underline Y$ is projective. Then 
  \[ 
    h^{i,j}(Y) = h^{j,i}(Y) 
    \quad \text{for all} \quad
    i,j \geq 0,
  \]
  where $h^{i,j}(Y) = \dim H^j(Y, \Omega^i_{Y/s})$ are the dimensions of the log Hodge cohomology groups.
\end{thm}

\begin{proof}
For the reader's convenience, we recall the proof. We may assume that $k=\CC$ and that $\underline Y$ is purely of dimension $n$. The log Betti cohomology $H^*(Y)$, endowed with $W_\bullet = $ the abutment filtration of the weight spectral sequence and ${\rm Fil}^\bullet =$ the Hodge filtration on $H^*(Y)\otimes \CC \simeq H^*(Y, \Omega^\bullet_{Y/s})$ is a mixed Hodge structure \cite{SteenbrinkLog}. For an ample line bundle $\mathcal{L}$ on $\underline Y$, the isomorphism from Theorem~\ref{thm:log-HL-char0}
\[ 
  L^r \colon H^{n-r}(Y) \isomto H^{n+r}(Y)(r)
\]
is a map of mixed Hodge structures, and hence strictly compatible with the Hodge filtrations. By \cite{IKN}, the spectral sequence
\[ 
  E_1^{i,j} = H^j(Y, \Omega^i_{Y/s}) 
  \quad \Rightarrow \quad 
  H^{i+j}(Y, \Omega^\bullet_{Y/s}) \simeq H^{i+j}(Y)\otimes \CC
\]
degenerates at $E_1$. Combining these, we obtain for $i+j=n-r$ an isomorphism
\[ 
  L^r \colon H^j(Y, \Omega^i_{Y/s}) = {\rm gr}^j H^{n-r}_{\rm dR}(Y/s)
  \isomto {\rm gr}^{j+r} H^{n+r}_{\rm dR}(Y/s) = H^{r+j}(Y, \Omega^{n-j}_{Y/s}).
\]
By log Serre duality \cite[(2.21)]{Tsuji}, the latter is dual to $H^{i}(Y, \Omega^{j}_{Y/s})$.
\end{proof}

\begin{thm} \label{thm:HS-char-0}
  Let $X$ be a smooth and proper rigid space over a complete discretely valued field $K\simeq k(\!(t)\!)$ with residue field $k$ of characteristic zero, admitting a formal model $\mathfrak{X}$ over $\cO_K\simeq k[\![t]\!]$ whose special fiber $Y = \mathfrak{X}_k$ is projective. Then 
  \[ 
    h^{i,j}(X) = h^{j,i}(X) 
    \quad \text{for all} \quad
    i,j \geq 0.
  \]
\end{thm}

\begin{proof} 
By resolution of singularities, after replacing $K$ with a finite extension $K'=k(\!(s)\!)$, $s^m = t$, we can assume that $X$ admits a proper and semistable model over $\cO_K$ whose special fiber is projective. Indeed, if $\mathfrak{X}$ is any proper flat model of $X$ whose special fiber $\mathfrak{X}_k$ is projective, resolution of singularities produces an admissible blow-up $\pi\colon \tilde{\mathfrak{X}}\to \mathfrak{X}$ with $\tilde{\mathfrak{X}}$ regular and $\tilde{\mathfrak{X}}_0$ a divisor with simple normal crossings. Since $\pi$ is a projective morphism, so is the map on special fibers $\pi_0\colon \tilde{\mathfrak{X}}_0 \to \mathfrak{X}_0$, and hence $\tilde{\mathfrak{X}}_0$ is projective. Let $m$ be the lowest common multiple of the multiplicities of the components of $\tilde{\mathfrak{X}}_0$, let $K' = k(\!(s)\!)$, $s^m = t$, and let $\tilde{\mathfrak{X}}' = \tilde{\mathfrak{X}} \otimes_{\cO_K} \cO_{K'}$. Note that the special fiber of $\tilde{\mathfrak{X}}'$ is the reduced closed subscheme of $\tilde{\mathfrak{X}}_0 \otimes_{\cO_K} \cO_{K'}$ and hence is projective. Now $\tilde{\mathfrak{X}}'$ may no longer be regular, but standard arguments (toric resolution of singularities) show that the minimal resolution $\mathfrak{X}\to \tilde{\mathfrak{X}}'$ is semistable over $\cO_{K'}$, and $\mathfrak{X}_0$ is projective for the same reason as before.

Let thus $f\colon \mathfrak{X} \to \operatorname{Spf} \cO_K$ be a proper semistable model of $X$ over $\cO_K$ whose special fiber $Y=\mathfrak{X}_0$ is projective. We endow $\operatorname{Spf} \cO_K$ and $\mathfrak{X}$ with the standard log structures, and $s=\Spec k$ and $Y$ with the induced log structures. This makes the map $f\colon \mathfrak{X}\to \operatorname{Spf} \cO_K$ log smooth and $Y$ into a strictly semistable log scheme over $k$. Let 
\[ 
  H^{i,j} = H^j(\mathfrak{X}, \Omega^i_{\mathfrak{X}/\cO_K}) = R^j f_* \Omega^i_{\mathfrak{X}/\cO_K}
\]
be the relative logarithmic Hodge cohomology groups. By Corollary~\ref{cor:hodge-free-formal}, the $H^{i,j}$ are free $\cO_K$-modules of finite rank, whose formation commutes with base change. Since we have $\operatorname{rank} H^{i,j} = \dim_K H^j(X, \Omega^i_X)$, to prove Hodge symmetry for $X$ it suffices to show the equalities
\[ 
  \dim_\CC H^{i,j}\otimes_{\cO_K} k = \dim_\CC H^{j,i}\otimes_{\cO_K} k.
\]
By the base change property, we have 
\[ 
  H^{i,j}\otimes_{\cO_K} k = H^j(Y, \Omega^i_{Y/s}).
\]
Therefore Hodge symmetry for $X$ is equivalent to log Hodge symmetry for $Y = \mathfrak{X}_0$, which holds by Theorem~\ref{thm:log-HS-char0}.
\end{proof}


\section{\texorpdfstring{Hodge symmetry in the $p$-adic situation}{Hodge symmetry in the p-adic situation}}
\label{s:p-adic}

In this section, we assume that $k$ has characteristic $p>0$ and denote by $H^*$ the log crystalline cohomology $H_{\text{log-cris}}^*(-/W(k))[1/p]$. It is a vector space over $K_0=W(k)[1/p]$ endowed with a Frobenius-linear isomorphism $\varphi$.

\begin{prop}[Log hard Lefschetz implies slope symmetry] \label{prop:log-HL-implies-SS}
  Let $Y$ be a strictly semistable log scheme over $k$ such that $\underline{Y}$ is proper and geometrically connected. Let $0\leq q \leq n=\dim \underline{Y}$ and let $\mathcal{L}$ be an ample line bundle on $\underline{Y}$. Let $\alpha_0\leq \ldots \leq \alpha_{s}$ ($s=\dim H^q(Y)-1$) be the slopes of Frobenius on $H^q(Y)$. If the log hard Lefschetz conjecture holds for $\mathcal{L}$ in degree $q$, then these slopes satisfy
  \[ 
    \alpha_{s - k} = q-\alpha_k.
  \]
\end{prop}

\begin{proof}
Since the Frobenius on $H^2(Y)$ multiplies $c_1(\mathcal{L})\in H^2(Y)$ by $p$ and the fundamental class in $H^{2n}(Y)$ by $p^n$, the Lefschetz operator induces a Frobenius-equivariant map
\[ 
  L^{n-q} \colon H^q(Y) \to H^{2n-q}(Y)(n-q) = \operatorname{Hom}(H^q(Y), H^{2n}(Y))(n-q) = H^q(Y)^\vee(-q),
\]
where $(-)(r)$ means crystalline Tate twist i.e.\ replacing the Frobenius $\varphi$ with $p^{-r}\varphi$. The slopes of $H^q(Y)^\vee$ are $-\alpha_{s}\leq \ldots \leq -\alpha_0$, and hence the slopes of the target are $q-\alpha_{s} \leq \ldots \leq q-\alpha_0$. Since $L^{n-q}$ is an isomorphism, we have $\alpha_{s-k} = q-\alpha_k$.   
\end{proof}

\begin{rmk} \label{rmk:suh}
Let $Y$ be a smooth and proper variety over a finite field $k$. By looking at the characteristic polynomial on Frobenius on the cohomology of $Y$, Suh \cite{Suh} has proved that the slopes of Frobenius on $H^q(Y)$ satisfy the assertion of Proposition~\ref{prop:log-HL-implies-SS}, even if $Y$ is not projective. 

Indeed, if $k=\FF_{p^e}$ and $\beta$ is an eigenvalue of $\varphi^e$ on $H^q(Y)$, then it is an algebraic integer such that
\[ 
  \beta \bar \beta = |\beta|^2 = p^{eq}, 
\] 
and hence $p^{eq}/\beta = \bar\beta$ is an eigenvalue as well. The slopes $\alpha_k$ equal the $p$-adic valuations of such roots $\beta$ multiplied by $1/e$, and hence their multiset is closed under $\alpha \mapsto q - \alpha$. (This observation appeared earlier in \cite[VI \S 3]{Ekedahl}.)

In particular, if $K$ is a finite extension of $\QQ_p$, then Theorem~\ref{thm:HS-p-adic}(\ref{thm:HS-good-ordinary}) below holds without the assumption that $Y$ is projective. We thank Alexander Petrov for this remark.
\end{rmk}

For the next result, we need to recall some $p$-adic Hodge theory \cite{Fontaine}. The $K_0$-vector space $H^q(Y) = H^q_{\text{log-cris}}(Y/W(k))[1/p]$ come equipped, in addition to the Frobenius $\varphi$, with a nilpotent operator $N$ satisfying $N\varphi = p\varphi N$. Further, the Hyodo--Kato isomorphism \cite[Theorem~5.1]{HyodoKato} (depending on a choice of a uniformizer of $\cO_K$)
\[ 
  H^q_{\text{log-cris}}(Y/W(k))[1/p] \otimes_{K_0} K \isomto H^q_{\rm dR}(X/K)
\]
endows $H^q(Y)\otimes K$ with a decreasing separated and exhaustive filtration ${\rm Fil}^\bullet$, obtained from the Hodge filtration on de Rham cohomology of $X$. The data $(\varphi, N, {\rm Fil}^\bullet)$ makes $H^q(Y)$ into a \emph{filtered $(\varphi, N)$-module} \cite[\S 4.3.2]{Fontaine}. 

For a filtered $(\varphi, N)$-module $D=(D, \varphi, N, {\rm Fil}^\bullet)$, one defines 
\[
  t_N(D) = \sum_{\alpha\in \QQ} (\dim_{K_0} D_\alpha)\cdot \alpha
  \quad \text{and} \quad
  t_H(D) = \sum_{i\in \ZZ} (\dim_K {\rm Fil}^i D/{\rm Fil}^{i+1} D)\cdot i,
\] 
where $D_\alpha\subseteq D$ is the part of slope $\alpha\in \QQ$. We say that $D$ is \emph{weakly admissible} \cite[\S 4.4.1]{Fontaine} if $t_N(D) = t_H(D)$ and $t_N(D')\geq t_H(D')$ for every subobject $D'\subseteq D$. It is a~consequence of $p$-adic Hodge theory (\cite{ColmezNiziol} in the case of formal schemes) that $H^q(Y)$ is weakly admissible.

\begin{prop} \label{prop:linear-rel-hij}
  Let $\mathfrak{X}$ be a strictly semistable proper formal scheme over $\cO_K$ of relative dimension $n$. Let $Y = \mathfrak{X}_k$ be its special fiber, endowed with the natural log structure, and let $X = \mathfrak{X}_K$ be its rigid analytic generic fiber. Let $0\leq q\leq n$, and suppose that $\underline{Y}$ admits an ample line bundle $\mathcal{L}$ for which the log hard Lefschetz holds in degree $q$. Then
  \[ 
    \sum_{i+j = q} (i-j) \cdot h^{i,j}(X)= 0.
  \]
  In particular, if $q\leq 2$ or $q\geq 2n-2$, then 
  \[ 
    h^{i,j}(X) = h^{j,i}(X) 
    \quad\text{for}\quad
    i+j =q.
  \]
\end{prop}

\begin{proof}
The weak admissibility of $H^q(Y)$ equipped with $\phi$, $N$, and ${\rm Fil}^\bullet$ implies that, if $\alpha_0\leq \ldots\leq \alpha_{s-1}$ are the Frobenius slopes on $H^q(Y)$, then
\[ 
  \sum_{i+j=q} h^{i,j}(X)\cdot i = t_H(H^q(Y)) = t_N(H^q(Y)) = \sum_{k=0}^{s-1} \alpha_k.  
\]
By Proposition~\ref{prop:log-HL-implies-SS}, this equals $\sum (q-\alpha_k) = q\cdot \dim H^q(Y) - \sum \alpha_k$. But
\begin{align*} 
  q\cdot \dim H^q(Y) - \sum_{k=0}^{s-1} \alpha_k &= q\sum_{i+j=q} h^{i,j}(Y) - \sum_{i+j=q} h^{i,j}(Y)\cdot i \\
  &= \sum_{i+j=q} h^{i,j}(Y)\cdot (q-i) =  \sum_{i+j=q} h^{i,j}(X)\cdot j. 
  \qedhere
\end{align*}
\end{proof}

\begin{defin}[{\cite[Definition~1.4]{PerrinRiouIllusie}}] \label{def:log-ordinary}
  Let $Y$ be a strictly semistable log scheme over $k$ such that $\underline{Y}$ is proper. We denote by $B\Omega^j_{Y/s} \subseteq \Omega^j_{Y/s}$ the sheaf of exact differential forms. We say that $Y$ is \emph{ordinary} if 
  \[ 
    H^i(Y, B\Omega^j_{Y/s}) = 0
    \quad \text{for all} \quad
    i,j \geq 0.
  \]
\end{defin}   

For $\underline{Y}$ is smooth, this is equivalent to the usual definition \cite[\S 7]{BlochKato}.

\begin{lemma}[{\cite[Proposition~1.5]{PerrinRiouIllusie}}] \label{lemma:ordinary-mod-tors}
  Let $Y$ be as in Definition~\ref{def:log-ordinary}. If $Y$ is ordinary, then the $F$-crystals
  \[ 
    H^q_{\text{log-cris}}(Y/W(k))/{\rm tors.} 
  \]
  are ordinary, i.e.\ have the same Newton and Hodge polygon, for all $q\geq 0$.
  The converse holds if $H^q_{\text{log-cris}}(Y/W(k))$ is torsion-free for all $q\geq 0$.
\end{lemma}

\begin{lemma}[{\cite[Proposition~1.6]{PerrinRiouIllusie}}] \label{lemma:Yi-Y-ordinary}
  Let $Y$ be a strictly semistable log scheme over $k$ such that $\underline{Y}$ is proper, and let $Y^{(k)}$ be as defined in \S\ref{ss:semistable-log}. If $Y^{(k)}$ are ordinary for all $k\geq 1$, then so is $Y$.
\end{lemma}

\begin{prop} \label{prop:X-Y-ordinary}
  Let $\mathfrak{X}$ be a proper and semistable formal scheme over $\cO_K$ with rigid generic fiber $X=\mathfrak{X}_K$ and log special fiber $Y = \mathfrak{X}_k$. Consider the following conditions.
  \begin{enumerate}[(a)]
    \item $Y$ is ordinary.
    \item $H^q_{\rm dR}(X/K)$, endowed with the structure of a weakly admissible filtered $(\varphi, N)$-module, is ordinary in the sense of \cite{PerrinRiou} for all $q\geq 0$, i.e.\ the Hodge polygon of $H^q_{\rm dR}(X/K)$ equals the Newton polygon for $H^q(Y)$.
    \item The $p$-adic Galois representations $H^q(X_{\widehat{\overline K}}, \QQ_p)$ are ordinary in the sense of \cite{PerrinRiou}, i.e.\ are iterated extensions of powers of the cyclotomic character.
  \end{enumerate}
  Then (a)$\Rightarrow$(b)$\Leftrightarrow$(c). Moreover, (b)$\Rightarrow$(a) if $H^*_{\text{log-cris}}(Y/W(k))$ is torsion-free and the Hodge cohomology groups $H^j(\mathfrak{X}, \Omega^i_{\mathfrak{X}/\cO_K})$ are free $\cO_K$-modules for all $i,j\geq 0$.
\end{prop}

\begin{proof}
Assume (a). By \cite[Corollaire~2.6]{PerrinRiouIllusie}, the Hodge filtration ${\rm Fil}^\bullet$ on $H^q_{\rm dR}(X/K)$ and the slope filtration $U_\bullet$ on $H^q(Y) = H^q_{\text{log-cris}}(Y/W(k))[1/p]$ are opposite in the sense that
\[ 
  (U_{i-1}\otimes_{W(k)[1/p]} K)\oplus {\rm Fil}^i = H^q_{\rm dR}(X/K)
\]
holds for all $i,q\geq 0$. This implies (b).

The equivalence of (b) and (c) follows from the semistable comparison theorem \cite{ColmezNiziol} and from \cite[Th\'eor\`eme~1.5]{PerrinRiou}.

For the final assertion, we observe that if Hodge cohomology is free, then the Hodge polygon of $H^q_{\rm dR}(X)$ equals the Hodge polygon defined by the Hodge numbers of $Y$. Since the crystalline cohomology is torsion-free, we can apply \cite[Proposition~1.6~(a)(iv)]{PerrinRiouIllusie} to conclude. 
\end{proof}

\begin{example}[Supersingular Enriques surfaces]
The following well-known example serves as a warning that some additional assumptions for the final assertion (b)$\Rightarrow$(a) in Proposition~\ref{prop:X-Y-ordinary} are necessary. 

Suppose $p=2$ and let $Y$ be a supersingular Enriques surface over $k$. Then $Y$ admits a~smooth lifting $\mathfrak{X}$ over a ramified extension of $W(k)$ \cite[Theorem~0.8]{EkedahlHylandShepherdBarron}. If $X$ is the generic fiber, then $H^*_{\rm dR}(X_{\overline K}/\overline K)$ is generated by algebraic cycles. Consequently, $h^{i,j}(X)$ vanishes for $i\neq j$ and the cohomology of $X$ is ordinary in the sense of Proposition~\ref{prop:X-Y-ordinary}(b). However, $Y$ is not ordinary, has torsion in $H^2_{\rm cris}$ and $H^3_{\rm cris}$, and the Hodge numbers of $Y$ differ from those of $X$ \cite[II 7.3]{IllusieDeRhamWitt}.
\end{example}

\begin{cor} \label{cor:ordinary-HS}
  In the context of Proposition~\ref{prop:linear-rel-hij}, suppose that the log special fiber $Y$ is ordinary. Then
  \[ 
    h^{i,j}(X) = h^{j,i}(X) 
    \quad\text{for}\quad
    i, j \geq 0.
  \]
\end{cor}

\begin{proof}
By Proposition~\ref{prop:X-Y-ordinary}, the Hodge number $h^{i,j}(X)$ equals the multiplicity of $i$ as a~Frobenius slope on $H^{i+j}(Y)$. We conclude by Proposition~\ref{prop:log-HL-implies-SS}.
\end{proof}

Combining the above results with the established cases of log hard Lefschetz, we obtain the following unconditional result. 

\begin{thm} \label{thm:HS-p-adic}
  Let $\mathfrak{X}$ be a proper and semistable formal scheme over $\cO_K$ with rigid generic fiber $X=\mathfrak{X}_K$ and log special fiber $Y = \mathfrak{X}_k$ which is projective. Then:
  \begin{enumerate}[(a)] 
    \item \label{thm:HS-good-ordinary} If $\mathfrak{X}$ is smooth over $\cO_K$ and $Y$ is ordinary, then $h^{i,j}(X) = h^{j,i}(X)$ for $i,j\geq 0$.
    \item \label{thm:HS-H1} $h^{1,0}(X) = h^{0,1}(X)$.
    \item \label{thm:HS-comb-red} If every connected component $Z$ of some $Y^{(k)}$ has torsion-free crystalline cohomology and admits a smooth projective lifting $\mathfrak{Z}$ over $\cO_K$ such whose cohomology is generated by algebraic cycles and such that the Hodge cohomology groups $H^i(\mathfrak{Z}, \Omega^j_{\mathfrak{Z}/\cO_K})$ are free $\cO_K$-modules for $i,j\geq 0$, then 
    \[ 
      h^{i,j}(X) = h^{j,i}(X)
      \quad \text{for} \quad 
      i,j\geq 0.
    \] 
  \end{enumerate}
\end{thm}

\begin{proof}
(a) Since in this case log hard Lefschetz amounts to the usual hard Lefschetz theorem, the statement follows from Proposition~\ref{cor:ordinary-HS}.

(b) Follows from Theorem~\ref{thm:log-HL-H1} and Proposition~\ref{prop:linear-rel-hij}.

(c) In light of Corollary~\ref{cor:ordinary-HS}, Proposition~\ref{lemma:Yi-Y-ordinary}, and Corollary~\ref{cor:log-HL-comb}, it suffices to prove that every component $Z$ of some $Y^{(k)}$ satisfies the following three conditions (see \S\ref{ss:alg-cycles} for the terminology regarding (A) and (B)):
\begin{enumerate}
  \item[(A)] The crystalline cohomology $H^*(Z)$ is generated by algebraic cycles, 
  \item[(B)] The Hodge standard conjecture holds for $Z$,
  \item[(C)] $Z$ is ordinary.
\end{enumerate}
We may assume that $k$ is algebraically closed. Let $\mathfrak{Z}$ be a smooth projective lifting of $Z$ over $\cO_K$ such that the geometric generic fiber $W = \mathfrak{Z}_{\overline K}$ has cohomology generated by algebraic cycles. Since $W$ lives in characteristic zero, it does not matter which cohomology theory we choose, and we prefer to use de Rham cohomology. We have specialization maps $A^i(W)\to A^i(Z)$ (see \cite[20.3.5]{Fulton}) fitting inside a commutative diagram 
\[ 
  \xymatrix{
    A^i(W)\ar[d] \ar[r] & A^i(Z)\ar[d] \\
    H^i_{\rm dR}(W/\overline K)  & H^i_{\rm cris}(Z/W(k))[1/p] \ar[l],
  }
\]
which after making everything $\overline K$-linear becomes
\[ 
  \xymatrix{
    A^i(W)\otimes_{\QQ} \overline K \ar[r] \ar[d] & A^i(Z)\otimes_{\QQ} \overline K\ar[d] \\
    H^i_{\rm dR}(W/\overline K) & H^i_{\rm cris}(Z/W(k))\otimes_{W(k)} \overline K. \ar[l]
  }
\]
Here, the left map is an isomorphism by assumption and the bottom map is an isomorphism by Berthelot--Ogus. Therefore the right map is surjective, showing (A). A~posteriori, all maps above are isomorphisms (see \S\ref{ss:alg-cycles}).

Using this with $i=1$, we see that a multiple of a given ample line bundle $\mathcal{L}$ lifts to $W$. Since $W$ is in characteristic zero, the Hodge standard conjecture holds for that lift of a multiple of $\mathcal{L}$, but since $A^*(W)\simeq A^*(Z)$, the Hodge standard conjecture holds for $Z$, proving (B).

Finally, since algebraic cohomology classes are of type $(i,i)$, the assumption on $W$ also shows that $h^{i,j}(W) = 0$ for $i\neq j$. Therefore the Hodge polygon for $H^i_{\rm dR}(W/\overline K)$ is a single segment, and hence it coincides with the Newton polygon of $H^i(Z)$. By the final assertion of Proposition~\ref{prop:X-Y-ordinary}, $Z$ is ordinary, showing (C). 
\end{proof}

\begin{rmk} \label{rmk:joshi}
In \cite{Joshi}, based on the symmetry of slope numbers \cite[VI \S 3]{Ekedahl}, Joshi shows that Hodge symmetry holds for a smooth and proper $Y$ over $k$ with degenerate Hodge spectral sequence and torsion-free crystalline cohomology which is Hodge--Witt (i.e.\ all $H^j(Y, W\Omega^i_Y)$ finitely generated). The Hodge--Witt property is weaker than ordinarity, and it is plausible that in Theorem~\ref{thm:HS-p-adic}(\ref{thm:HS-good-ordinary}) one can replace \emph{ordinary} by \emph{Hodge--Witt}.
\end{rmk}

\begin{example}[Cellular varieties] \label{ex:cellular}
Though the assumptions of Theorem~\ref{thm:HS-p-adic}(\ref{thm:HS-comb-red}) may look a bit discouraging, they are satisfied in many situations of geometric interest. Recall that a smooth proper variety $Z$ is \emph{cellular} if it admits a stratification $Z=\coprod_{\alpha\in I} Z_\alpha$ where each $Z_\alpha$ is isomorphic to some affine space $\Aff^{n_\alpha}_k$. If $I_k\subseteq I$ ($0\leq k\leq \dim Z$) denotes the set of strata of codimension $k$, then the cycle classes of the closures $\overline{Z_\alpha}$ ($\alpha\in I_k$) give an isomorphism
\[ 
  {\rm CH}^k(Z) = \bigoplus_{\alpha\in I_k} \ZZ \cdot [\overline{Z_\alpha}] 
  \quad 
  \text{(see \cite[Lemma~3.3]{Kahn})}.
\]
One can upgrade it to a decomposition of Chow motives
\[ 
  M(Z) = \bigoplus_{k=0}^n {\rm CH}^k(Z)\otimes \ZZ(k)[2k],
  \quad
  \text{\cite[Corollary~3.5]{Kahn}}.
\]
Consequently, for each of the cohomology theories
\[ 
  H^* \in \{\text{$\ell$-adic, crystalline, de Rham, Hodge, Hodge--Witt} \}
\]
one has a decomposition
\[ 
  H^*(Z) = \bigoplus_{k=0}^n {\rm CH}^k(Z)\otimes H^0({\rm pt})[2k],
\]
see \cite[Corollary~2.8]{AchingerZdanowicz}. In particular, the crystalline cohomology of $Z$ is torsion-free and $Z$ is ordinary.

Suppose now that $Z$ admits a \emph{cellular lift} $\mathfrak{Z}$ over $\cO_K$, i.e.\ a smooth lifting together with a stratification $\mathfrak{Z} = \coprod_{i\in I} \mathfrak{Z}_\alpha$ where $\mathfrak{Z}_\alpha \simeq \Aff^{n_\alpha}_{\cO_K}$ with $(\mathfrak{Z}_\alpha)_k = Z_\alpha$. Then the above assertions (of course, with the exception of crystalline and Hodge--Witt cohomology) hold for the geometric generic fiber $W = \mathfrak{Z}_{\overline K}$. This implies in particular that $H^j(\mathfrak{Z}, \Omega^i_{\mathfrak{Z}/\cO_K})$ are free.

Well-known examples of varieties admitting a cellular decomposition include toric varieties, rational homogeneous spaces, or more generally every smooth projective variety admitting an algebraic torus action with finitely many fixed points \cite{BialynickiBirula}. 
\end{example}


\bibliographystyle{amsalpha} 
\bibliography{bib}

\end{document}